\documentclass[12pt, reqno]{amsart}
\usepackage[T1]{fontenc}
\usepackage[utf8]{inputenc}
\usepackage[english]{babel}
\usepackage{amsmath,amsfonts,amssymb,amsthm,mathrsfs}
\usepackage{accents,mathtools,extarrows}
\usepackage{bbm}
\usepackage{yhmath}
\usepackage{graphicx}
\usepackage{rotating}
\usepackage{wasysym}
\usepackage{textcomp}
\usepackage{hyperref,cleveref}
\usepackage{xfrac}
\usepackage{enumitem}
\usepackage{tikz-cd}									
\usetikzlibrary{matrix,arrows,calc,positioning,decorations.pathmorphing}
\tikzset{
  shift left/.style ={commutative diagrams/shift left={#1}},
  shift right/.style={commutative diagrams/shift right={#1}}
}
\usepackage{todonotes}

\theoremstyle{plain}
\newtheorem{theorem}{Theorem}
\newtheorem{lemma}[theorem]{Lemma}
\newtheorem{proposition}[theorem]{Proposition}

\newtheorem*{theorem*}{Theorem}

\theoremstyle{definition}
\newtheorem{definition}[theorem]{Definition}

\newtheorem{remark}[theorem]{Remark}

\newtheorem*{thexreftheorem}{Theorem \thexref}

\newtheorem*{thexrefcorollary}{Corollary \thexcoro}

\newcommand{\N}{\mathbb N}
\newcommand{\Z}{\mathbb Z}

\newcommand{\R}{\mathbb R}

\renewcommand{\epsilon}{\varepsilon}
\renewcommand{\phi}{\varphi}

\begin{document}

\title{A criterion for slope 1 homological stability}
\author{Mikala Ørsnes Jansen and Jeremy Miller}\thanks{Jeremy Miller was supported by NSF Grant DMS-2202943. Mikala Ørsnes Jansen was supported by the Danish National Research Foundation through the Copenhagen Centre for Geometry and Topology (DNRF151).}

\begin{abstract}
We show that for nice enough $\N$-graded $\mathbb{E}_2$-algebras, a diagonal vanishing line in $\mathbb{E}_1$-homology of gives rise to slope $1$ homological stability. This is an integral version of a result by Kupers--Miller--Patzt.
\end{abstract}

\maketitle

The $\mathbb{E}_k$-cellular approach to homological stability \cite{KupersMiller,GalatiusKupersRandalWilliams} allows one to prove homological stability results by studying $\mathbb{E}_k$-cell structures on $\mathbb{E}_k$-algebras. In particular, in many situations it is possible to establish a vanishing line for $\mathbb{E}_k$-homology by studying the equivariant homology of certain simplicial complexes \cite{GalatiusKupersRandalWilliams18b, GalatiusKupersRandalWilliams20,KupersMillerPatzt}. Generalizing arugments appearing in \cite{GalatiusKupersRandalWilliams18b}, \cite[Proposition 5.1]{KupersMillerPatzt} gives a criterion for slope $1$ homological stability for $\mathbb{E}_3$-algebras with $\Z[\sfrac{1}{2}]$-coefficients given a vanishing line for $\mathbb{E}_1$-homology. The purpose of this note is to remove the hypotheses that $2$ is a unit in the coefficient ring and generalize the result to $\mathbb{E}_2$-algebras.


Throughout, we will use notation from \cite{GalatiusKupersRandalWilliams}. Our main result is the following.

\begin{theorem}\label{prop}
Let $\Bbbk$ be a commutative ring and $\mathbf{A}\in \operatorname{Alg}_{\mathbb{E}_2}(\operatorname{sMod}_\Bbbk^{\N})$ with $H_{\ast,0}(\overline{\mathbf{A}})\cong \Bbbk[\sigma]$ for $\sigma\in H_{1,0}(\mathbf{A})$. If $H_{n,d}^{\mathbb{E}_1}(\mathbf{A})=0$ for all $d<n$, then $H_{n,d}(\overline{\mathbf{A}}/\sigma)=0$ for all $d<n$.
\end{theorem}

This result will be applied in \cite{OrsnesJansen24} to deduce slope $1$ homological stability (with integer coefficients) for $\operatorname{RBS}$-categories over Euclidean domains. The $\operatorname{RBS}$-categories were introduced in \cite{ClausenOrsnesJansen} as a model for unstable algebraic K-theory. Such a stability result has also been announced in \cite[Remark 1.4]{Randal-Williams24}, but that proof will (to the best of our understanding) rely on very different methods. It will also be used by the second author and R. Sroka to establish slope $1$ stability for certain Iwahori-Hecke algebras.

Note that \Cref{prop} is false if $\mathbf{A}$ is only assumed to be an $\mathbb{E}_1$-algebra. The primary difficulty in proving \Cref{prop} is that the conclusions and hypotheses of the theorem do not apply to free algebras (see \Cref{RemarkFree}).


We begin with the following lemma, which informally tells us that attaching $\mathbb{E}_2$-cells of high enough slope cannot give rise to homological stability that was not already there (compare for example \cite[Proposition 6.1]{KupersMillerPatzt} and \cite[Theorem 2.2]{MillerPatztPetersenRandalWilliams}).

\begin{lemma}\label{lemma}
Let $\Bbbk$ be a commutative ring and $\mathbf{A}\in \operatorname{Alg}_{\mathbb{E}_2}(\operatorname{sMod}_\Bbbk^{\N})$ with $H_{\ast,0}(\overline{\mathbf{A}})\cong \Bbbk[\sigma]$ for $\sigma\in H_{1,0}(\mathbf{A})$. Suppose $\mathbf{B}\in \operatorname{Alg}_{\mathbb{E}_2}(\operatorname{sMod}_\Bbbk^{\N})$ is another $\mathbb{E}_2$-algebra equipped with a morphism $f\colon \mathbf{A}\rightarrow \mathbf{B}$ that is an isomorphism on $0'$th homology and an epimorphism on $1'$st homology. Assume moreover that
\begin{align*}
H_{n,d}^{\mathbb{E}_2}(\mathbf{B},\mathbf{A})=0\qquad \text{for all }d<n.
\end{align*}
If $H_{n,d}(\overline{\mathbf{B}}/\sigma)=0$ for $d<n$, then $H_{n,d}(\overline{\mathbf{A}}/\sigma)=0$ for $d<n$
\end{lemma}
\begin{proof}
First of all, we note that by \cite[Theorem 15.9]{GalatiusKupersRandalWilliams}, we have
\begin{align*}
H_{n,d}^{\overline{\mathbf{A}}}(\overline{\mathbf{B}},\overline{\mathbf{A}})=0\qquad \text{for all }d<n.
\end{align*}
Moreover, \cite[Corollary 11.14]{GalatiusKupersRandalWilliams} applies to show that we additionally have
\begin{align*}
H_{n,d}^{\overline{\mathbf{A}}}(\overline{\mathbf{B}},\overline{\mathbf{A}})=0\qquad \text{whenever }d\leq 1.
\end{align*}

It follows that there is a CW-approximation of $\overline{\mathbf{A}}$-modules
\begin{center}
\begin{tikzpicture}
\matrix (m) [matrix of math nodes,row sep=2em,column sep=2em]
  {
\overline{\mathbf{A}} & \overline{\mathbf{B}} \\
\mathbf{M} & \\
  };
  \path[-stealth]
(m-1-1) edge node[above]{$f$} (m-1-2)
(m-1-1) edge (m-2-1)
(m-2-1) edge node[below right]{$\scriptstyle\sim$} (m-1-2)
;
\end{tikzpicture}
\end{center}
that only has relative $(n,d)$-cells for $d\geq n$ and $d\geq 2$ (\cite[Theorem 11.21]{GalatiusKupersRandalWilliams}). The skeletal filtration of $\mathbf{M}$ induces a filtration of the cofibre $\mathbf{M}/\overline{\mathbf{A}}$ with associated graded
\begin{align*}
\operatorname{gr}(\mathbf{M}/\overline{\mathbf{A}})\simeq \bigoplus_{\alpha}S_\Bbbk^{n_\alpha,d_\alpha} \otimes \overline{\mathbf{A}}\qquad \text{with }d_\alpha\geq n_\alpha \text{ and } d_\alpha\geq 2.
\end{align*}

Since the functor $(-)/\sigma = (-) \otimes_{\overline{\mathbf{A}}} \overline{\mathbf{A}}/\sigma$ preserves homotopy cofibre sequences of $\overline{\mathbf{A}}$-modules (\cite[Definition 12.14]{GalatiusKupersRandalWilliams}), the filtration of $\mathbf{M}/\overline{\mathbf{A}}$ passes to a filtration of $(\mathbf{M}/\overline{\mathbf{A}})/\sigma$ whose associated spectral sequence has $E^1$-page:
\begin{align}\label{ss}
E^1_{n,p,q}=\bigoplus_{\substack{\alpha \\ d_\alpha=p}}H_{n,p+q}(S_\Bbbk^{n_\alpha,d_\alpha} \otimes \overline{\mathbf{A}}/\sigma) \ \Rightarrow\  H_{n,p+q}((\mathbf{M}/\overline{\mathbf{A}})/\sigma).
\end{align}
Note here that
\begin{align*}
H_{n,d}(S_\Bbbk^{n_\alpha,d_\alpha} \otimes \overline{\mathbf{A}}/\sigma)\cong H_{n-n_\alpha,d-d_\alpha}(\overline{\mathbf{A}}/\sigma)
\end{align*}
by the Künneth spectral sequence.

Assume now that $H_{n,d}(\overline{\mathbf{B}}/\sigma)=0$ for all $d<n$. We'll prove by induction on $d$ that
$H_{n,d}(\overline{\mathbf{A}}/\sigma)=0$ for all $d<n$. We observe first that
\begin{align*}
H_{n,0}(\overline{\mathbf{A}}/\sigma)=0 \qquad \text{for }n>0
\end{align*}
as $H_{n,0}(\overline{\mathbf{A}})\cong \Bbbk[\sigma]$ by assumption.

Now let $D<n$ and assume that the claim holds for all $d<D$. For any $d\leq D+1$ and any $\alpha$ from the indexing set of the skeletal filtration of $\mathbf{M}$, we have $d-d_\alpha<D$ (as $d_\alpha\geq 2$). Thus by the induction hypothesis, the spectral sequence (\ref{ss}) has $E^1_{n,p,q}=0$ whenever $p+q\leq D+1$ and we conclude that
\begin{align*}
H_{n,d}((\mathbf{M}/\overline{\mathbf{A}})/\sigma)\qquad \text{for any }d\leq D+1.
\end{align*}

Consider the cofibre sequence
\begin{align*}
\overline{\mathbf{A}}/\sigma\rightarrow \mathbf{M}/\sigma\rightarrow (\mathbf{M}/\overline{\mathbf{A}})/\sigma
\end{align*}
and the following small section of the associated long exact sequence in homology
\begin{align*}
H_{n,D+1}((\mathbf{M}/\overline{\mathbf{A}})/\sigma) \rightarrow H_{n,D}(\overline{\mathbf{A}}/\sigma)\rightarrow H_{n,D}(\mathbf{M}/\sigma).
\end{align*}
Combining the above with the assumption on the homology of $\mathbf{M}/\sigma\simeq \overline{\mathbf{B}}/\sigma$, we conclude that
\begin{align*}
H_{n,D}(\overline{\mathbf{A}}/\sigma)=0
\end{align*}
as desired. That finishes the proof.
\end{proof}

Conversely, if one starts with an $\mathbb{E}_2$-algebra with a given slope of stability, attaching cells of high enough slope preserves the given stability. This is a general strategy in this business and will also be the proof strategy of our main result. Since \Cref{prop} does not apply to free algebras, we will need to start our stability argument with a slightly more exotic $\mathbb{E}_2$-algebra.


\begin{definition}
Consider the free $\mathbb{E}_2$-algebra on a single generator $\sigma$ in bidegree $(1,0)$:
\begin{align*}
\mathbf{E}_2(S^{1,0}_\Bbbk\sigma)
\end{align*}
and the element
\begin{align*}
Q^1\sigma \in H_{2,1}(\mathbf{E}_2(S^{1,0}_\Bbbk\sigma))\cong \Z
\end{align*}
(see \cite[\S 16]{GalatiusKupersRandalWilliams}). Choose a map $Q^1\sigma\colon S^{2,1}_\Bbbk \rightarrow \mathbf{E}_2(S^{1,0}_\Bbbk\sigma)$ representing this class and construct the $\mathbb{E}_2$-algebra
\begin{align*}
\mathbf{X}:= \mathbf{E}_2(S^{1,0}_\Bbbk\sigma) \cup^{\mathbb{E}_2}_{Q^1\sigma} D^{2,2}_\Bbbk\rho
\end{align*}
by attaching a cell along $Q^1\sigma$.
\end{definition}

\begin{proposition} \label{Xstability}
$H_{n,d}(\overline{\mathbf{X}}/\sigma)=0$ for $d<n$.
\end{proposition}

\begin{proof}
Observe that the $\mathbb{E}_2$-algebra $\mathbf{X}$ is homotopy equivalent to the space $\operatorname{Sym}^{\leq 2}(\R^2)$ of \cite[Definition 6.1]{KupersMiller}: this follows directly from \cite[Proposition 6.11]{KupersMiller} and the observation that the $\mathcal{O}$-cells in \cite{KupersMiller}, which are defined using partial algebras, coincide with the $\mathcal{O}$-cells of \cite{GalatiusKupersRandalWilliams} (see \cite[Remark 2.6]{KlangKupersMiller}).

It is observed on p.44 of \cite{KupersMiller} that $\operatorname{Sym}^{\leq 2}(\R^2)$ satisfies homological stability of slope $1$: more precisely, the stabilisation map $\operatorname{Sym}_n^{\leq 2}(\R^2)\rightarrow \operatorname{Sym}_{n+1}^{\leq 2}(\R^2)$ induces an isomorphism
\begin{align*}
H_d(\operatorname{Sym}_n^{\leq 2}(\R^2))\xrightarrow{\ \cong \ } H_d(\operatorname{Sym}_{n+1}^{\leq 2}(\R^2)), \quad \text{for }d\leq n.
\end{align*}
In particular, $H_{n,d}(\overline{\mathbf{X}}/\sigma)=0$ for $d<n$.
\end{proof}

\begin{lemma}\label{lemma2}
Consider the $\mathbb{E}_2$-algebra
\begin{align*}
\mathbf{Y}:= \mathbf{E}_2(S^{1,0}_\Bbbk \sigma \oplus S^{1,1}_\Bbbk \alpha) \cup^{\mathbb{E}_2}_{Q^1\sigma-\sigma \alpha} D^{2,2}_\Bbbk \rho
\end{align*}
with trivial cells $\sigma$ and $\alpha$ in bidegrees $(1,0)$, respectively $(1,1)$, and a cell $\rho$ in bidegree $(2,2)$ attached along $Q^1\sigma-\sigma \alpha$. This satisfies 
\begin{align*}
H_{n,d}(\overline{\mathbf{Y}}/\sigma)=0\qquad \text{for all }d<n.
\end{align*}
\end{lemma}
\begin{proof}
Let $\mathbf{X}'$ be the $\mathbb{E}_2$-algebra obtained from $\mathbf{Y}$ by attaching a cell in bidegree $(1,2)$ glued along the generator $\alpha$. Then the map of $\mathbb{E}_2$-algebras
\begin{align*}
\mathbf{E}_2(S^{1,0}_\Bbbk\sigma)\rightarrow \mathbf{X}'
\end{align*}
extends to a map $\mathbf{X}\rightarrow \mathbf{X}'$. In fact, this map is a homotopy equivalence since the cells cancel out (see the proof of \cite[Proposition 6.3]{KupersMillerPatzt} for a similar strategy).

Thus we have maps of $\mathbb{E}_2$-algebras
\begin{align*}
\mathbf{Y}\longrightarrow \mathbf{X}'\xleftarrow{\ \simeq \ }\mathbf{X}.
\end{align*}
By \Cref{Xstability}, we have
\begin{align*}
H_{n,d}(\mathbf{X}'/\sigma)=0\qquad \text{for }d<n,
\end{align*}
and \Cref{lemma} implies that this also holds for $\mathbf{Y}$ as we've attached a cell in bidegree $(1,2)$.
\end{proof}

We now turn to the proof of our main result.

\begin{proof}[Proof of \Cref{prop}]
We will first of all show that the stabilisation map
\begin{align*}
\sigma\cdot \colon H_{1,1}(\mathbf{A})\rightarrow H_{2,1}(\mathbf{A})
\end{align*}
is surjective (this can be viewed as a converse to the claim on p. 201 of \cite{GalatiusKupersRandalWilliams} and in fact we'll use much the same techniques to prove it). Consider the free $\mathbb{E}_1$-algebra $\mathbf{E}_1(\mathbf{A}(1))$ and the map $\mathbf{E}_1(\mathbf{A}(1))\rightarrow \mathbf{A}$. From the long exact sequence in $\mathbb{E}_1$-homology,
\begin{align*}
\cdots\rightarrow H_{n,d}^{\mathbb{E}_1}(\mathbf{E}_1(\mathbf{A}(1)))\rightarrow H_{n,d}^{\mathbb{E}_1}(\mathbf{A})\rightarrow H_{n,d}^{\mathbb{E}_1}(\mathbf{A},\mathbf{E}_1(\mathbf{A}(1)))\rightarrow H_{n,d-1}^{\mathbb{E}_1}(\mathbf{E}_1(\mathbf{A}(1)))\rightarrow \cdots,
\end{align*}
we deduce that the groups
\begin{align*}
H_{n,d}^{\mathbb{E}_1}(\mathbf{A},\mathbf{E}_1(\mathbf{A}(1)))
\end{align*}
vanish for $(n,d)\in \{(1,0),(1,1),(2,0),(2,1)\}$ (here we use that $H_{2,\ast}^{\mathbb{E}_1}(\mathbf{E}_1(\mathbf{A}(1)))$ always vanishes, that $H_{2,0}^{\mathbb{E}_1}(\mathbf{A})=0$ by assumption and that the map $\mathbf{E}_1(\mathbf{A}(1))\rightarrow \mathbf{A}$ induces an isomorphism on $H_{1,0}^{\mathbb{E}_1}(-)$ and a surjection on $H_{1,1}^{\mathbb{E}_1}(-)$). It follows from \cite[Corollary 11.14]{GalatiusKupersRandalWilliams} that the same holds for the relative homology groups $H_{n,d}(\mathbf{A},\mathbf{E}_1(\mathbf{A}(1)))$. In particular
\begin{align*}
H_{2,1}(\mathbf{A},\mathbf{E}_1(\mathbf{A}(1)))=0.
\end{align*}
Plugging this into the long exact sequence in homology, we see that the map
\begin{align*}
H_{2,1}(\mathbf{E}_1(\mathbf{A}(1)))\longrightarrow H_{2,1}(\mathbf{A})
\end{align*}
is a surjection. Recall that $\mathbf{E}_1(\mathbf{A}(1))$ is equivalent to the free associative algebra $ \bigoplus_{i\geq 0} \mathbf{A}(1)^{\otimes i}$; hence, from the Künneth isomorphism
\begin{align*}
H_{2,1}(\mathbf{E}_1(\mathbf{A}(1)))\cong H_0(\mathbf{A}(1))\otimes H_1(\mathbf{A}(1))\oplus H_1(\mathbf{A}(1))\otimes H_0(\mathbf{A}(1))
\end{align*}
we deduce that the stabilisation map $H_1(\mathbf{A}(1))\rightarrow H_1(\mathbf{A}(2))$ is a surjection as desired.

Let $\sigma\colon S^{1,0}_\Bbbk\rightarrow \mathbf{A}$ denote a representative of the generator $\sigma$ in $H_{1,0}(\mathbf{A})$. This extends to a map of $\mathbb{E}_2$-algebras
\begin{align*}
f\colon \mathbf{E}_2(S^{1,0}_\Bbbk\sigma)\rightarrow \mathbf{A}.
\end{align*}

Choose a map $Q^1\sigma\colon S^{2,1}_\Bbbk \rightarrow \mathbf{E}_2(S^{1,0}_\Bbbk\sigma)$ representing the class $Q^1\sigma \in H_{2,1}(\mathbf{E}_2(S^{1,0}_\Bbbk\sigma))$. Since the map $\sigma\cdot \colon H_{1,1}(\mathbf{A})\rightarrow H_{2,1}(\mathbf{A})$ is surjective, there is an $\alpha\in H_{1,1}(\mathbf{A})$ such that $Q^1\sigma=\sigma\alpha$ in $H_{2,1}(\mathbf{A})$. A choice of null-homotopy of the map $Q^1\sigma-\sigma\alpha\colon S^{2,1}_\Bbbk \rightarrow \mathbf{A}$ gives rise to a map of $\mathbb{E}_2$-algebras
\begin{align*}
g\colon \mathbf{Y}=\mathbf{E}_2(S^{1,0}_\Bbbk \sigma\oplus S^{1,1}_\Bbbk \alpha)\cup^{\mathbb{E}_2}_{Q^1\sigma-\sigma\alpha}D^{2,2}_\Bbbk \rho \rightarrow \mathbf{A}
\end{align*}
where we readily identify the $\mathbb{E}_2$-algebra $\mathbf{Y}$ constructed in the previous lemma.

Now recall that we can transfer the vanishing line in $\mathbb{E}_1$-homology to $\mathbb{E}_2$-homology (\cite[Theorem 14.4]{GalatiusKupersRandalWilliams}); in other words,
\begin{align*}
H_{n,d}^{\mathbb{E}_2}(\mathbf{A})=0\quad \text{for }d<n.
\end{align*}
Consider the associated long exact sequence in $\mathbb{E}_2$-homology
\begin{align*}
\cdots\rightarrow H_{n,d}^{\mathbb{E}_2}(\mathbf{Y})\rightarrow H_{n,d}^{\mathbb{E}_2}(\mathbf{A})\rightarrow H_{n,d}^{\mathbb{E}_2}(\mathbf{A},\mathbf{Y})\rightarrow H_{n,d-1}^{\mathbb{E}_2}(\mathbf{Y})\rightarrow \cdots
\end{align*}

Combining the vanishing line $H_{n,d}^{\mathbb{E}_2}(\mathbf{A})=0$, $d<n$, with the fact that $H_{n,d}^{\mathbb{E}_2}(\mathbf{Y})$ is non-trivial only when $(n,d)$ is $(1,0)$, $(1,1)$ or $(2,2)$, we see that $H_{n,d}^{\mathbb{E}_2}(\mathbf{A},\mathbf{Y})=0$ for $d<n$. It follows that there is a CW-approximation
\begin{center}
\begin{tikzpicture}
\matrix (m) [matrix of math nodes,row sep=2em,column sep=2em]
  {
\mathbf{Y} & \mathbf{A} \\
\mathbf{C} & \\
  };
  \path[-stealth]
(m-1-1) edge node[above]{$g$} (m-1-2)
(m-1-1) edge (m-2-1)
(m-2-1) edge node[below right]{$\scriptstyle\sim$} (m-1-2)
;
\end{tikzpicture}
\end{center}
where $\mathbf{C}$ is obtained from $\mathbf{Y}$ by attaching $\mathbb{E}_2$-cells in bidegrees $(n,d)$ satisfying $d\geq n$.

Attaching $\mathbb{E}_2$-cells of slope $\sfrac{d}{n}\geq 1$ preserves this slope of stability: indeed by \cite[Corollary 15.10]{GalatiusKupersRandalWilliams}, we have the following vanishing line for module cells
\begin{align*}
H_{n,d}^{\overline{\mathbf{Y}}}(\overline{\mathbf{A}})=0\qquad \text{for }d<n.
\end{align*}
It follows that we have a CW-approximation of $\overline{\mathbf{Y}}$-modules 
\begin{align*}
\mathbf{D}\xrightarrow{\ \sim\ } \overline{\mathbf{A}}
\end{align*}
such that $\mathbf{D}$ only has $(n,d)$-cells with $d\geq n$ (\cite[Theorem 11.21]{GalatiusKupersRandalWilliams}). Since the functor $(-)/\sigma = (-) \otimes_{\overline{\mathbf{Y}}} \overline{\mathbf{Y}}/\sigma$ preserves homotopy cofibre sequences of $\overline{\mathbf{Y}}$-modules (\cite[Definition 12.14]{GalatiusKupersRandalWilliams}), the skeletal filtration of $\mathbf{D}$ passes to a filtration of $\mathbf{D}/\sigma$ with associated graded
\begin{align*}
\operatorname{gr}(\mathbf{D}/\sigma)\simeq \operatorname{gr}(\mathbf{D})/\sigma\simeq\bigoplus_{\alpha}S_\Bbbk^{n_\alpha,d_\alpha} \otimes \overline{\mathbf{Y}}/\sigma
\end{align*}
with $n_\alpha\geq d_\alpha$. Consider the Künneth spectral sequence
\begin{align*}
E^2_{n,p,q}=\bigoplus_{\substack{n'+n''=n,\\ q'+q''=q}}\operatorname{Tor}_p^\Bbbk \bigg(H_{n',q'}\big(\bigoplus_{\alpha}S_\Bbbk^{n_\alpha,d_\alpha}\big), H_{n'',q''}(\overline{\mathbf{Y}}/\sigma)\bigg) \ \Rightarrow \ H_{n,p+q}\big(\operatorname{gr}(\mathbf{D}/\sigma)\big)
\end{align*}
(\cite[Lemma 10.5]{GalatiusKupersRandalWilliams}). Note that the group $H_{n',q'}\big(\bigoplus_{\alpha}S_\Bbbk^{n_\alpha,d_\alpha}\big)$ vanishes whenever $q'<n'$. Hence, it can only be trivial for $n'\leq q'$. Assuming $p+q<n$, we see that $n'\leq q'$ implies that $n''-q''\geq n-q\geq n-q-p>0$ and hence
$H_{n'',d''}(\overline{\mathbf{Y}}/\sigma)=0$ by \Cref{lemma2}. It follows that $E^2_{n,p,q}=0$ whenever $p+q<n$ and hence
\begin{align*}
H_{n,d}(\operatorname{gr}(\mathbf{D}/\sigma))=0\qquad \text{for }d<n.
\end{align*}
By the skeletal spectral sequence (\cite[Theorem 10.10]{GalatiusKupersRandalWilliams}), it follows that the same vanishing line holds for $\mathbf{D}/\sigma$:
\begin{align*}
H_{n,d}(\overline{\mathbf{A}}/\sigma)\cong H_{n,d}(\mathbf{D}/\sigma) =0\quad \text{for }d<n
\end{align*}
as desired.
\end{proof}

\begin{remark}
We remark that the exact same strategy can be used to prove slope $\tfrac{m}{m+1}$ stability as in \cite[Proposition 5.1]{KupersMillerPatzt}. In \Cref{prop}, assume that
\begin{align*}
H^{\mathbb{E}_2}_{n,d}(\mathbf{A})=0\quad \text{for }d<n-1\quad \text{and}\quad H^{\mathbb{E}_2}_{2,1}(\mathbf{A})=H^{\mathbb{E}_2}_{3,2}(\mathbf{A})=\cdots =H^{\mathbb{E}_2}_{m,m+1}(\mathbf{A}) =0.
\end{align*}
Then
\begin{align*}
H_{n,d}(\overline{\mathbf{A}}/\sigma)=0\qquad\text{for } d\leq \tfrac{m}{m+1}(n-1).
\end{align*}
We leave the details to the reader, so the ideas of the proof stand out more clearly.

Likewise \Cref{lemma} can be generalised to slope $\tfrac{m}{m+1}$ stability for a statement as in \cite[Proposition 6.1]{KupersMillerPatzt}: assume instead that $H^{\mathbb{E}_2}_{n,d}(\mathbf{B},\mathbf{A})=0$ for $d\geq \tfrac{m}{m+1}n+1$. Then
\begin{align*}
H_{n,d}(\overline{\mathbf{B}}/\sigma)=0\qquad\text{for } d\leq \tfrac{m}{m+1}(n-1)
\end{align*}
will imply the same vanishing line for homology of $\overline{\mathbf{A}}/\sigma$.
\end{remark}

\begin{remark} \label{RemarkFree}
In \cite{KupersMillerPatzt}, they establish \Cref{prop} under the stronger hypothesis that $k>2$ and $2$ is a unit in $\Bbbk$. Their proof strategy is similar to ours except with $\mathbf{Y}$ replaced by the free $\mathbb{E}_k$-algebra $\mathbf{E}_k(S^{1,0}_\Bbbk \sigma)$. The homology groups $H_{n,i}(\mathbf{E}_k(S^{1,0}_\Bbbk \sigma))$ are completely known by work of Cohen \cite{CohenLadaMay} from which one sees that $H_{n,i}(\mathbf{E}_k(S^{1,0}_\Bbbk \sigma))$ has slope $1$ stability (with the stated offset) if and only if $k>2$ and $2$ is a unit in $\Bbbk$. In particular,
\begin{align*}
H_{2,1}(\overline{\mathbf{E}}_k(S^{1,0}_\Z \sigma)/\sigma)=\Z\quad\text{if } k=2\quad \text{and}\quad H_{2,1}(\overline{\mathbf{E}}_k(S^{1,0}_\Z \sigma)/\sigma)=\Z /2\Z \quad\text{if }k>2.
\end{align*}
Thus, the stated hypotheses are necessary if one wants to prove slope $1$ stability for an algebra $\mathbf{A}$ by attaching cells to $\mathbf{E}_k(S^{1,0}_\Bbbk \sigma)$ and using, as we have done, that cell attachments preserve homological stability. Our primary innovation thus consists of building $\mathbf{A}$ by attaching cells to $\mathbf{Y}$ instead.
\end{remark}

\bibliographystyle{alpha}

\end{document}